\documentclass[a4paper,10pt]{scrartcl}
\usepackage{bbm}
\usepackage{amsmath,amsthm}
\newtheorem{df}{Definition}[section]
\newtheorem{prop}[df]{Proposition}
\newtheorem{lem}[df]{Lemma}
\newtheorem{thm}[df]{Theorem}

\def\Z{\mathbbm Z}
\def\Q{\mathbbm Q}
\def\id{\text{id}}
\title{Quotients of the topology of the partition lattice which are not homotopy equivalent to wedges of spheres}
\author{Ralf Donau\\Universit\"at Bremen\\\texttt{ruelle@math.uni-bremen.de}}

\begin{document}

\maketitle

\begin{abstract}
The reader of \cite{donau} might conjecture that $\Delta(\overline{\Pi}_n)/G$ is homotopy equivalent to a wedge of spheres for any $n\geq 3$ and any subgroup $G\subset S_n$. We disprove this by showing that $\Delta(\overline{\Pi}_p)/C_p$ is not homotopy equivalent to a wedge of spheres for any prime number $p\geq 5$.
\end{abstract}
\section{Introduction}
Let $n\geq 3$ and let $\Pi_n$ denote the poset consisting of all partitions of $[n]:=\{1,\dots,n\}$ ordered by refinement, such that the finer partition is the smaller partition. Let $\overline{\Pi}_n$ denote the poset obtained from $\Pi_n$ by removing both the smallest and greatest element, which are $\{\{1\},\dots,\{n\}\}$ and $\{[n]\}$, respectively. We consider $\overline{\Pi}_n$ as a category, which is acyclic, and define $\Delta(\overline{\Pi}_n)$ to the nerve of the acyclic category $\overline{\Pi}_n$, which is a regular trisp, see \cite[Chapter 10]{buch}.

The symmetric group $S_n$ operates on $\Delta(\overline{\Pi}_n)$ in a natural way. We set $S_1\times S_{n-1}:=\{\sigma\in S_n\mid\sigma(1)=1\}$.
\begin{thm}[Donau \cite{donau}]
\label{rdthm}
Let $n\geq3$ and $G\subset S_1\times S_{n-1}$ be an arbitrary subgroup, then the topological space $\Delta(\overline{\Pi}_n)/G$ is homotopy equivalent to a wedge of spheres of dimension $n-3$.
\end{thm}
\section{Free actions on trisps}
Let $\Delta$ be a trisp and $G$ a group that acts freely on $\Delta$, then the quotient map $\Delta\longrightarrow\Delta/G$ is a covering space by Proposition 1.40 and Exercise 23 in \cite{hatcher}. Assume $\Delta$ is simply connected, then $\pi_1(\Delta/G)$ is isomorphic to $G$ by Proposition 1.40 in \cite{hatcher}. Furthermore, if $G$ is additionally an abelian group, then $H_1(\Delta/G;\Z)$ is isomorphic to $G$. We obtain the following results:
\begin{prop}
Let $\Delta$ be a simply connected trisp and $G$ a finite group that acts freely on $\Delta$, then $\Delta/G$ is not homotopy equivalent to a wedge of spheres.
\end{prop}
\begin{lem}
\label{FreeAct}
Let $\Delta$ be a simply connected trisp and $G$ an abelian group that acts freely on $\Delta$, then $H_1(\Delta/G;\Z)$ is isomorphic to $G$.
\end{lem}
\begin{df}
\emph{Euler characteristic}
\[
\chi(\Delta):=\sum_{n\geq 0}(-1)^n|S_n(\Delta)|
\]
$S_n(\Delta)$ denotes the set of simplices of dimension $n$.
\end{df}
\begin{df}
\emph{Betti number}
\[
\beta_n^F(\Delta):=\dim H_n(\Delta;F)
\]
\end{df}
\begin{thm}
\label{euler}
Let $\Delta$ be a trisp which has finitely many simplices. Then we have
\[
\chi(\Delta)=\sum_{n\geq 0}(-1)^n\beta_n^F(\Delta)
\]
for any field $F$.
\end{thm}
The proof can be found in \cite[Chapter 3]{buch}.
\begin{lem}
\label{Bettis}
Let $\Delta$ be a trisp that is homotopy equivalent to a wedge of $k$ spheres of dimension $d>0$. Let $G$ be a finite group that acts freely on $\Delta$. Then
\[
H_i(\Delta/G;\Q)\cong
\begin{cases}
\Q&\text{for $i=0$}\\
\Q^{\frac {k+1}{|G|}-1}&\text{for $i=d$ and $d$ even}\\
\Q^{\frac {k-1}{|G|}+1}&\text{for $i=d$ and $d$ odd}\\
0&\text{else}
\end{cases}
\]
\end{lem}
\begin{proof}
We have $H_i(\Delta;\Q)\cong 0$ for $i\not=0$ and $i\not=d$. By applying the Transfer Theorem\footnote{See \cite{transfer}} we obtain $H_i(\Delta/G;\Q)\cong 0$ for $i\not=0$ and $i\not=d$. By Theorem \ref{euler} we have
\[\chi(\Delta)=\sum_{n\geq 0}(-1)^n\beta_n^\Q(\Delta)=1+(-1)^dk\]
which implies $\chi(\Delta/G)=\frac{1+(-1)^dk}{|G|}$, since $G$ acts freely. We apply Theorem \ref{euler} again and obtain $\chi(\Delta/G)=1+(-1)^d\beta_d^\Q(\Delta/G)$. Hence $\beta_d^\Q(\Delta/G)=(-1)^d(\frac{1+(-1)^dk}{|G|}-1)=\frac{(-1)^d+k}{|G|}-(-1)^d$.
\end{proof}
\section{A prime period action on the reduced subset lattice}
Now let $p\geq 5$ be a prime number. We consider the subgroup of $S_p$ that is generated by a cycle of length $p$, which we denote by $C_p$. Set
\[L_p:={\cal P}([p])\setminus\{\emptyset,[p]\}\]
Here, ${\cal P}([p])$ denotes the set of subsets of $[p]$. $S_p$ acts on $L_p$ in a natural way.
\begin{lem}
\label{FreeZPS}
Let $p>0$ be a prime number, then $C_p$ acts freely on $L_p$. In particular $C_p$ acts freely on $\Delta(L_p)$.
\end{lem}
\begin{proof}
We have to show $gv=v$ implies $g=\id$ for all $v\in L_p$ and $g\in C_p$. Let $g\in C_p$ with $g\not=\id$, then $g$ generates $C_p$. Let $v\in L_p$ and assume $gv=v$. Since $C_p$ acts transitively on the set $[p]$, $v\not=\emptyset$ implies $v=[p]$, which is impossible.
\end{proof}
\begin{prop}
Let $p\geq 5$ be a prime number, then
\begin{eqnarray}
H_1(\Delta(L_p)/C_p;\Z)&\cong&\Z_p\\
H_{p-2}(\Delta(L_p)/C_p;\Z)&\cong&\Z
\end{eqnarray}
\end{prop}
\begin{proof}
$C_p$ acts freely on $\Delta(L_p)$ by Lemma \ref{FreeZPS}. By applying Lemma \ref{FreeAct} we obtain $H_1(\Delta(L_p)/C_p;\Z)\cong\Z_p$, since $\Delta(L_p)$ is simply connected. The abstract simplicial complex $L_p$, where the vertices are the singleton sets, is homotopy equivalent to a sphere of dimension $p-2$. This can be verified via Discrete Morse Theory for example. Since $\Delta(L_p)$ is the barycentric subdivision of $L_p$, $\Delta(L_p)$ is also homotopy equivalent to a sphere of dimension $p-2$. Via Lemma \ref{Bettis} we obtain $H_{p-2}(\Delta(L_p)/C_p;\Q)\cong\Q$.
\end{proof}
\section{A prime period action on the reduced partition lattice}
Let $n\geq 3$ be a fixed natural number. The symmetric group $S_n$ operates on $\Delta(\overline{\Pi}_n)$ in a natural way. Let $p\geq 5$ be a prime number, then $\Delta(\overline{\Pi}_p)$ is homotopy equivalent to a wedge of $(p-1)!$ spheres of dimension $(p-3)$ by Theorem \ref{rdthm}. We consider the quotient $\Delta(\overline{\Pi}_p)/C_p$.
\begin{lem}
\label{FreeZPP}
Let $p\geq 3$ be a prime number, then $C_p$ acts freely on $\Delta(\overline{\Pi}_p)$.
\end{lem}
\begin{proof}
It suffices to show that $C_p$ acts freely on the set of vertices of $\Delta(\overline{\Pi}_p)$. We have to show $gv=v$ implies $g=\id$ for all $v\in\overline{\Pi}_p$ and $g\in C_p$. Assume we have $gv=v$ with $g\not=\id$. Since $g$ generates $C_p$, we have $gv=v$ for all $g\in\Z_p$. In particular, $C_p$ acts on the blocks of $v$ and this action is free by Lemma \ref{FreeZPS}. Hence we have at least $p$ blocks, which is impossible.
\end{proof}
\begin{prop}
Let $p\geq 5$ be a prime number, then
\begin{eqnarray}
H_1(\Delta(\overline{\Pi}_p)/C_p;\Z)&\cong&\Z_p\\
H_{p-3}(\Delta(\overline{\Pi}_p)/C_p;\Z)&\cong&\Z^{\frac{(p-1)!-(p-1)}p}
\end{eqnarray}
In particular $\Delta(\overline{\Pi}_p)/C_p$ is not homotopy equivalent to a wedge of spheres.
\end{prop}
\begin{proof}
$C_p$ acts freely on $\Delta(\overline{\Pi}_p)$ by Lemma \ref{FreeZPP}. By applying Lemma \ref{FreeAct} we obtain $H_1(\Delta(\overline{\Pi}_p)/C_p;\Z)\cong\Z_p$, since $\Delta(\overline{\Pi}_p)$ is simply connected. Via Lemma \ref{Bettis} we obtain $H_{p-3}(\Delta(\overline{\Pi}_p)/C_p;\Q)\cong\Q^{\frac{(p-1)!+1}p-1}\cong\Q^{\frac{(p-1)!-(p-1)}p}$.
\end{proof}
\section*{Acknowledgements}
The author would like to thank Dmitry N. Kozlov for the valuable suggestions that led to improvements of this paper.

\end{document}